\pdfoutput=1
\RequirePackage{ifpdf}
\ifpdf 
\documentclass[pdftex]{sigma}
\else
\documentclass{sigma}
\fi

\numberwithin{equation}{section}

\newtheorem{teiri}{Theorem}[section]
\newtheorem{prop}[teiri]{Proposition}
\newtheorem{kei}[teiri]{Corollary}

\theoremstyle{definition}
\newtheorem{rei}[teiri]{Example}
\newtheorem{teigi}[teiri]{Definition}

\begin{document}

\allowdisplaybreaks

\newcommand{\arXivNumber}{1701.00931}

\renewcommand{\PaperNumber}{017}

\FirstPageHeading

\ShortArticleName{Klein's Fundamental 2-Form of Second Kind for the $C_{ab}$ Curves}

\ArticleName{Klein's Fundamental 2-Form of Second Kind\\ for the $\boldsymbol{C_{ab}}$ Curves}

\Author{Joe SUZUKI}

\AuthorNameForHeading{J.~Suzuki}

\Address{Department of Mathematics, Osaka University, Machikaneyama Toyonaka,\\ Osaka 560-0043, Japan}
\Email{\href{mailto:suzuki@math.sci.osaka-u.ac.jp}{suzuki@math.sci.osaka-u.ac.jp}}

\ArticleDates{Received January 05, 2017, in f\/inal form March 11, 2017; Published online March 16, 2017}

\Abstract{In this paper, we derive the exact formula of Klein's fundamental 2-form of second kind for the so-called $C_{ab}$ curves. The problem was initially solved by Klein in the 19th century for the hyper-elliptic curves, but little progress had been seen for its extension for more than 100 years. Recently, it has been addressed by several authors, and was solved for subclasses of the $C_{ab}$ curves whereas they found a way to f\/ind its individual solution numerically. The formula gives a standard cohomological basis for the curves, and has many applications in algebraic geometry, physics, and applied mathematics, not just analyzing sigma functions in a general way.}


\Classification{14H42; 14H50; 14H55}

\section{Introduction}

Since Abel, Jacobi, Poincar\'e, and Riemann established its framework, theories of Abelian and modular functions associated with algebraic curves have been of crucial importance in algebraic geometry, physics, and applied mathematics. Algebraic curves in this paper are intended as compact Riemann surfaces.

The study of the hyper-elliptic curves goes back to the beginning of the 20th century, and these appear in much detail in advanced text-books such as Baker~\cite{baker} and Cassels and Flynn~\cite{cf}. However, little has been considered for more general curves than the hyper-elliptic curves. In this paper, we consider a class of curves ($C_{ab}$ curves) in the form
\begin{gather}\label{eq1}
\sum_{i,j}c_{i,j}x^iy^j=0,
\end{gather}
where $(i,j)$ range over $ai+by\leq ab$, $0\leq i\leq b$, $0\leq j\leq a$ for some mutually prime integers $a$, $b$, and $c_{i,j}$ are constants such that~(\ref{eq1}) implies $ \sum\limits_{i,j}ic_{i,j}x^{i-1}y^j\not=0$ or $ \sum\limits_{i,j}jc_{i,j}x^{i}y^{j-1}\not=0$.

This paper studies cohomologies of $C_{ab}$ curves in terms of which the Klein--Weierstrass construction of multivariate Abelian and sigma functions is made possible. It does not seek further theories but focuses on the description of cohomologies of these curves.

According to the Riemann--Roch theorem, the entire holomorphic dif\/ferentials make a vector space over $\mathbb C$ of dimension $g$. For example, we can take them as
\begin{gather}\label{eq97}
 du_i(x,y)=\frac{x^{i-1}dx}{2y} ,
\end{gather}
$i=1,2,\dots,g$ for the hyper-elliptic curve
\begin{gather}\label{eq3}
y^2=\sum_{i=0}^{2g+1}c_ix^i
\end{gather}
with $c_{i,0}=c_i$, $c_{i,1}=0$, $i=0,\dots,2g+1$.

We consider a 2-form $((x,y),(z,w))\mapsto R((x,y),(z,w))dxdz$ over $C\times C$ that has a (double) pole only at $(z,w)=(x,y)$ and is normalized:
\begin{gather*}\lim_{(z,w)\rightarrow (x,y)}R((x,y),(z,w))(x-z)^2=1.\end{gather*}
Such a 2-form can be written as
\begin{gather}\label{eq8}
R((x,y),(z,w)):=\frac{d}{dz}\Omega((x,y),(z,w))+\sum_{i=1}^g\frac{du_i(x,y)}{dx}\frac{dr_i(z,w)}{dz},
\end{gather}
using dif\/ferentials $\{dr_i\}_{i=1}^g$ of the second kind that have poles only at inf\/inity and a meromorphic function $\Omega((x,y),(z,w))$. If we further require symmetry:
\begin{gather*}R((x,y),(z,w))=R((z,w),(x,y)) ,\end{gather*}
then such an $\{dr_i\}$ is unique modulo operation of ${\rm Sp}(2g,{\mathbb Z})$ and the space generated by $\{du_i\}$ (see~\cite{b} for hyper-elliptic curves and \cite{BB, AA} for general curves). We call such a normalized symmetric $R((x,y),(z,w))dxdz$ Klein's fundamental 2-form of second kind~\cite{klein1,klein2}.

For example, for the hyper-elliptic curve (\ref{eq3}), such a 2-form can be obtained using the canonical meromorphic dif\/ferentials $\{dr_i\}_{i=1}^g$
\begin{gather}\label{eq99}
dr_i(z,w)=\sum_{k=i}^{2g-i}c_{k+1+i}(k+1-i)\frac{z^{k}}{2w}dz
\end{gather}
and 1-form
\begin{gather*}\Omega((x,y),(z,w)):=\frac{y+w}{2(x-z)y}dx .\end{gather*}

It is known that the normalized symmetric bilinear $R((x,y),(z,w))$ can be expressed using~(\ref{eq8}) for~$C_{ab}$~\cite{nakayashiki} and telescopic curves~\cite{ayano} although $\Omega((x,y),(z,w))$ needs more extensions. However, for the canonical meromorphic dif\/ferentials $\{dr_i\}_{i=1}^g$, no exact formula has been given. For example, $C_{34}$ and $C_{35}$ curves and $C_{45}$ curves were studied in~\cite{DD} and~\cite{FF}, respectively.
Recently, the reference \cite{GG} addressed $C_{3,g+1}$ curves with genus $g$. These results dealt with a~specif\/ic class of curves and failed to obtain the formula for general $C_{ab}$ curves with arbitrary mutually prime~$a$,~$b$. As a~result, still many researchers are either numerically calculating or algorithmically computing $\{dr_i\}_{i=1}^g$ given $\{c_{i,j}\}$ and $\{du_i\}_{i=1}^g$. This paper extends~(\ref{eq99}) and obtains the general formula in a closed form for the $C_{ab}$ curves. It contains all the existing results and has many applications including algebraic expression of sigma functions~\cite{nakayashiki}, def\/ining equations of the Jacobian varieties, etc.~\cite{b,onishi}.

This paper is organized as follows: Section~\ref{section2} sets up the holomorphic dif\/ferentials $\{du_i\}_{i=1}^g$ and gives background of this paper.
Section~\ref{section3} states and proves the main result (formula) and gives two typical examples both of which extend the previous cases. In the last section, we raise an open problem.

\section{Background}\label{section2}

Let $a$, $b$ be mutually prime positive integers, and $C$ the curve def\/ined by
\begin{gather*}F(x,y):=\sum_{i,j}c_{i,j}x^iy^j=0\end{gather*}
with a unique point $\cal O$ at which the zero orders of $x$ and $y$ are $a$ and $b$, respectively, where $(i,j)$ range over
\begin{gather*}D:=\{(i,j)\,|\,ai+bj\leq ab,\, 0\leq i\leq b,\, 0\leq j\leq a\},\end{gather*}
and $c_{i,j}$ are constants such that either $ \sum\limits_{i,j} ic_{i,j}x^{i-1}y^j\not=0$ or $ \sum\limits_{i,j}jc_{i,j}x^iy^{j-1}\not=0$ for each $(x,y)$ on the curve.

We consider the set of 1-forms \cite{nakayashiki}
\begin{gather*}du_{i,j}=\frac{x^iy^j}{ \frac{\partial F}{\partial y}(x,y)}dx,\end{gather*}
where $(i,j)$ range over
\begin{gather*}J(a,b):=\{(i,j)\,|\,i,j\geq 0,\, ai+bj\leq ab-a-b\}.\end{gather*}

We know by the general theory that for $g$ variable points $(x_1,y_1), \dots,(x_g,y_g)$ on $C$, the sum of integrals from $\cal O$ to those $g$ points
\begin{gather*}u=(u_{i,j})_{(i,j)\in J(a,b)}=\sum_{(i,j)\in J(a,b)}\int_{\cal O}^{(x_i,y_i)}du_{i,j}\end{gather*}
f\/ill the whole space ${\mathbb C}^g$, where the weights of the variables $u_{i,j}$ are $ab-a(i+1)-b(j+1)$. If we regard the weight of each coef\/f\/icient $c_{i,j}$ in (\ref{eq1}) is $ai+bi-ab$, the weights assigned to the dif\/ferentials render $F(x,y)$ homogeneous.

\begin{prop}\label{proposition1}
The set $J(a,b)$ has
\begin{gather*}\#J(a,b)=g=\frac{(a-1)(b-1)}{2}\end{gather*}
elements, and the zero orders of $\{du_{i,j}\}$ are nonnegative and different.
\end{prop}

\begin{teigi}
The 2-form $R((x,y),(z,w))dxdz$ on $C\times C$ is called the fundamental 2-form of the second kind if the following conditions are satisf\/ied
\begin{enumerate}\itemsep=0pt
\item[1)] it is symmetric: $R((x,y),(z,w))=R((z,w),(x,y))$,
\item[2)] it has its only pole along the diagonal of $C\times C$, and
\item[3)] in the vicinity of each point, it is expanded in power series as
\begin{gather*}R((x,y),(z,w))dxdz=\frac{dt_{xy}dt_{zw}}{(t_{x,y}-t_{z,w})^2}+O(1)\end{gather*}
as $(z,w)\rightarrow (x,y)$, where $t_{x,y}$ and $t_{z,w}$ are local coordinates of points $(x,y)$ and $(z,w)$, respectively.
\end{enumerate}
\end{teigi}

We shall look for a realization of $R((x,y),(z,w))$ in the form
\begin{gather*}R((x,y),(z,w))=\frac{G((x,y),(z,w))}{(x-z)^2\frac{\partial F}{\partial y}(x,y)\frac{\partial F}{\partial w}(z,w)},\end{gather*}
where $G((x,y),(z,w))$ is a symmetric polynomial in its variables.

\begin{rei}[hyper-elliptic curves \cite{CC,EE}] For the hyper-elliptic curve (\ref{eq3}), in which $a=2$ and $b=2g+1$, where $g$ is the genus of the curve $C$, then (\ref{eq1}) expresses a~hyper-elliptic curve, the meromorphic function on $C\times C$
\begin{gather*}\Omega((x,y),(z,w))=\frac{y+w}{2y(x-z)}\end{gather*}
and dif\/ferentials (\ref{eq97}), there exists $\{dr_{i}\}_{i=1}^g$ such that (\ref{eq8}) is symmetric. In fact,
\begin{gather}
\frac{d}{dz}\Omega((x,y),(z,w))=\frac{ \frac{dw}{dz}(x-z)+(y+w)}{2y(x-z)^2}=\frac{(x-z)\sum\limits_{k=1}^{2g+1}kc_kz^{k-1}+2w(y+w)}{4yw(x-z)^2}\nonumber\\
\hphantom{\frac{d}{dz}\Omega((x,y),(z,w))}{} =
\frac{ x\sum\limits_{k=1}^{2g+1}kc_kz^{k-1}-\sum\limits_{k=1}^{2g+1}(k-2)c_kz^i+2c_0+2yw}{4yw(x-z)^2}.\label{eq199}
\end{gather}
If we add
\begin{gather*}(x-z)^2\sum_{k=1}^gx^{k-1}\sum_{h=2k+1}^{2g+1}c_h(h-2k)z^{h-k-1}\end{gather*}
to the numerator of (\ref{eq199}), we obtain
\begin{gather*}G((x,y),(z,w))=2\sum_{k\colon \, {\rm even}}c_k2x^{k/2}z^{k/2}+\sum_{k\colon \, {\rm odd}} c_k\big(x^{\frac{k+1}{2}}z^{\frac{k-1}{2}}+x^{\frac{k-1}{2}}z^{\frac{k+1}{2}}\big) +2c_0+2yw,
\end{gather*}
where
\begin{gather*}
-(k-2)z^k+kz^{k-1}x+(x-z)^2\sum_{h=1}^{i/2}(k-2h)z^{k-h-1}x^{h-1}\\
\qquad{} =-(k-2l)z^{k-l+1}x^{l-1}+(k-2l+2)z^{k-l}x^l+(x-z)^2\sum_{h=l}^{\lfloor k/2\rfloor }(k-2h)z^{k-h-1}x^{h-1}
\end{gather*}
for $l=1,2,\dots,\lfloor i/2 \rfloor$ has been applied ($\lfloor k/2 \rfloor=k/2$ and $(k-1)/2$ when $k$ is even and odd, respectively), which means that by choosing
\begin{gather*}\frac{dr_i}{dz}(z,w)=\sum_{h=2i+1}^{2g+1}\frac{(h-2i)z^{h-i-1}}{2w},\end{gather*}
we obtain
\begin{gather*}R((x,y),(z,w))=R((z,w),(x,y))\end{gather*}
and
\begin{gather*}\lim_{(z,w)\rightarrow (x,y)}R((x,y),(z,w))(x-z)^2=\lim_{(z,w)\rightarrow (x,y)}\frac{G((x,y),(z,w))}{4yw}=1.\end{gather*}
\end{rei}

\section[The fundamental 2-form of the second kind for the $C_{ab}$ curves]{The fundamental 2-form of the second kind for the $\boldsymbol{C_{ab}}$ curves}\label{section3}

Hereafter, we denote $ F_z:=\frac{\partial F(z,w)}{\partial z}$, $ F_w:=\frac{\partial F(z,w)}{\partial w}$, $ H:=\frac{F(y,z)-F(w,z)}{y-z}$, $ H_z:=\frac{\partial H}{\partial z}$, and $ H_w:=\frac{\partial H}{\partial w}$.

We f\/ind symmetric $R((x,y),(z,w))$ for the meromorphic function
\begin{gather*}\Omega((x,y),(z,w)):=\frac{H}{(x-z)F_y}.\end{gather*}

To this end, if we note $ \frac{dw}{dz}=-\frac{F_z}{F_w}$, we have
\begin{gather}\label{eq919}
\frac{d\Omega}{dz}=\frac{(H_z+H_w\frac{dw}{dz})(x-z)+H}{(x-z)^2F_y}=\frac{(H_zF_w-H_wF_z)(x-z)+HF_w}{(x-z)^2F_yF_w}.
\end{gather}

Let $ g_j:=\sum\limits_{i\colon \, 0\leq ai+bj\leq ab}c_{i,j}z^i$, $ g_j':=\frac{\partial g_j}{\partial z}$, and $ h_j:=\sum\limits_{i=0}^{j-1}w^iy^{j-1-i}$ for $j=0,\dots,a$. Then, one checks
\begin{gather}\label{eq16}
F_w=\sum_{j=0}jw^{j-1}g_j,\qquad H=\sum_{j=0}^a h_jg_j,\qquad H_z=\sum_{j=0}^a h_jg_j',\qquad H_w=\sum_{j=0}^a \frac{\partial h_j}{\partial w}g_j.
\end{gather}
Let
\begin{gather*}I(i,j):=\left(-w^i\frac{\partial h_j}{\partial w}+jw^{j-1}h_i- \frac{\partial {h}_{i}}{\partial y}w^j\right)g_i'g_j(x-z)+\big(jw^{j-1}h_i-jh_{j-1}w^i\big)g_ig_j\end{gather*}
for $i,j=0,\dots,a$. Then, from (\ref{eq16}) and that the arithmetic is modulo $F(z,w)=\sum\limits_{j=0}^a w^jg_j=0$, we have
\begin{gather*}\sum_{i=0}^a\sum_{j=0}^a I(i,j)=\sum_{i=0}^a\left\{-w^ig_i'H_w+h_ig_i'F_w-\frac{\partial h_i}{\partial y}g_i'\sum_{j=0}^a w^jg^j\right\}(x-z)\\
\hphantom{\sum_{i=0}^a\sum_{j=0}^a I(i,j)=}{} +
\sum_{i=0}^a\left\{h_ig_iF_w-w^ig_i\sum_{j=0}^ajh_{j-1}g_j\right\},\end{gather*}
which coincides with the numerator of (\ref{eq919}).

We seek $ \{\frac{dr_{i,j}}{dz}\}_{(i,j)\in J(a,b)}$ such that the 2-form
\begin{gather*}R((x,y),(z,w)):=\frac{d\Omega_{(z,w)}}{dz}+\sum_{(i,j)\in J(a,b)}\frac{du_{i,j}}{dx}\frac{dr_{i,j}}{dz}\end{gather*}
is symmetric, in other words, its numerator
\begin{gather*}\sum_{u=0}^{a}\sum_{v=0}^{a}
I(u,v)+ (x-z)^2\sum_{(i,j)\in J(a,b)}u_{i,j}(x,y)r_{i,j}(z,w)\end{gather*}
is symmetric, where
\begin{gather*}\frac{dr_{i,j}}{dz}=\frac{r_{i,j}(z,w)}{F_w}\end{gather*}
is the dif\/ferential of the second kind given the dif\/ferentials of the f\/irst kind{\samepage
\begin{gather*}\frac{du_{i,j}}{dx}=\frac{u_{i,j}(x,y)}{F_y}\end{gather*}
with $u_{i,j}(x,y)=x^iy^j$.}

Let $\underline{\rm mod}(\alpha,\beta):=\alpha-\lfloor\alpha/\beta\rfloor*\beta$, and $\overline{\rm mod}(\alpha,\beta):=
\begin{cases}
\underline{\rm mod}(\alpha,\beta)&\underline{\rm mod}(\alpha,\beta)\not=0,\\
\alpha&\underline{\rm mod}(\alpha,\beta)=0
\end{cases}$
for positive integers $\alpha$, $\beta$, and $\lceil\gamma\rceil$ and $\lfloor\gamma\rfloor$ are the smallest integer no less than $\gamma$ and the largest integer no more than $\gamma$, respectively for positive real~$\gamma$. For example, $\overline{\rm mod}(3,2)=\underline{\rm mod}(3,2)=2$, $2=\overline{\rm mod}(4,2)\not=\underline{\rm mod}(4,2)=0$, etc.

Then, the following property plays an important role in the derivation of the main theorem:

\begin{prop}\label{proposition2} If $m\leq n$, then
\begin{gather}
I(m,n)+I(n,m) =mw^{m-1}y^{n-1}g_mg_n'(x-z)+mw^{m-1}y^{n-1}g_mg_n\nonumber\\
\hphantom{I(m,n)+I(n,m) =}{} +mw^{n-1}y^{m-1}g_m'g_n(x-z)+mw^{n-1}y^{m-1}g_mg_n \nonumber\\
\hphantom{I(m,n)+I(n,m) =}{}
-\sum_{k=m+1}^{n-1}\big[\{(n-k)g_mg_n'+(k-m)g_m'g_n\}(x-z)+(n-m)g_mg_n\big]\nonumber\\
\hphantom{I(m,n)+I(n,m) =}{}\times w^{k-1}y^{m+n-k-1}. \label{eq401}
\end{gather}
\end{prop}

\begin{teiri}\label{theorem1}
$R((x,y),(z,w))$ is symmetric when
\begin{gather}
 r_{i,j}(z,w) =\sum_{u\leq j}\sum_{r}\sum_{s}c_{r,u}c_{s,j+1}u(s-i-1)z^{r+s-i-2}w^{u-1}\label{eq21}\\
\hphantom{r_{i,j}(z,w) =}{} +\sum_{j+1\leq v}\sum_{r}\sum_{s}c_{r,j+1}c_{s,v}(j+1)(r-i-1)z^{r+s-i-2}w^{v-1}\label{eq22}\\
\hphantom{r_{i,j}(z,w) =}{} +\sum_{u\leq j}\sum_{j+2\leq v}\sum_{r}\sum_{s}c_{r,u}c_{s,v}\{(i+1)(v-u)-(j+1-u)s\nonumber\\
\hphantom{r_{i,j}(z,w) =+\sum_{u\leq j}\sum_{j+2\leq v}\sum_{r}\sum_{s}}{} -(v-j-1)r\}z^{r+s-i-2}w^{u+v-j-2}, \label{eq91}
\end{gather}
where $(r,s)$ range over $s\geq i+2$ in $D$ for the first term, over $r\geq i+2$ in $D$ for the second term, and over $(j+1-u)s+(v-j-1)r\leq (i+1)(v-u)$ and $r+s\geq i+2$ in $D$ for the last term. The symmetric value is given by $ R((x,y),(z,w))=\frac{G((x,y),(z,w))}{(x-z)^2\frac{\partial F(x,y)}{\partial x}\frac{\partial F(z,w)}{\partial z}}$ with
\begin{gather}G((x,y),(z,w))=\sum_{u=0}^aG_{u,u}^{(12)}
+\sum_{u=0}^a\sum_{u< v}\big\{G_{u,v}^{(12)}+G_{u,v}^{(34)}+G^{(5)}_{u,v}\big\},\nonumber\\ \label{eq24}
G^{(12)}_{u,v}:=u\sum_r\sum_sc_{r,u}c_{s,v}x^rz^sy^{u-1}w^{v-1},
\\ \label{eq25}
G^{(34)}_{u,v}:=u\sum_r\sum_sc_{r,u}c_{s,v}x^sz^ry^{v-1}w^{u-1},
\end{gather}
and
\begin{gather}
 G^{(5)}_{u,v} :=-\sum_{k=u+1}^{v-1}\sum_r\sum_sc_{r,u}c_{s,v}
\big\{
\overline{\rm mod}\{(k-u)s+(v-k)r,v-u\}
z^{ \lceil\frac{(k-u)s+(v-k)r}{v-u}\rceil}
\nonumber\\
\hphantom{G^{(5)}_{u,v} :=}{}
\times x^{ \lfloor\frac{(k-u)r+(v-k)s}{v-u}\rfloor}+
\underline{\rm mod}\{(k-u)r+(v-k)s,v-u\}
z^{ \lfloor\frac{(k-u)s+(v-k)r}{v-u}\rfloor}
\nonumber\\
\hphantom{G^{(5)}_{u,v} :=}{}
\times x^{ \lceil\frac{(k-u)r+(v-k)s}{v-u}\rceil}\big\} w^{k-1}y^{u+v-k-1}. \label{eq181}
\end{gather}
\end{teiri}

\begin{proof} First of all, we prove
\begin{gather*}u\sum_r\sum_s c_{r,u}c_{s,v}\big\{t^{(12)}(x,z)+(x-z)^2\Delta^{(12)}(x,z)\big\}w^{u-1}y^{v-1}=G^{(12)}_{u,v}\end{gather*}
for $0\leq u\leq v\leq a$, where
\begin{gather*} t^{(12)}(x,z):=sz^{r+s-1}(x-z)+z^{r+s}\end{gather*}
and
\begin{gather*}\Delta^{(12)}(x,z):=\sum_{k=1}^{s-1}kz^{r+k-1}x^{s-k-1}.\end{gather*}
In fact, we see
\begin{gather*}t^{(12)}(x,z)+(x-z)^2\Delta^{(12)}(x,z)\\
\qquad{} =-(s-l)z^{r+s-l+1}x^{l-1}+(s-l+1)z^{r+s-l}x^l+(x-z)^2\sum_{k=1}^{s-l}kz^{r+k-1}x^{s-k-1}\end{gather*}
for $l=1,\dots,s$. In particular,
\begin{gather*}\sum_r\sum_s uc_{r,u}c_{s,v}\Delta^{(12)}(x,z)y^{v-1}w^{u-1}\\
\qquad{} =\sum_{(i,j)\in J(a,b)}x^iy^j\sum_r\sum_s uc_{r,u}c_{s,j+1}(s-i-1)z^{r+s-i-2}w^{u-1},\end{gather*}
where $(r,s)$ range over $s\geq i+1$ in $D$. Thus, from the f\/irst and second terms in (\ref{eq401}), we obtain~(\ref{eq21}) and~(\ref{eq24}).

Similarly,
\begin{gather*}u\sum_r\sum_s c_{r,u}c_{s,v}\big\{t^{(34)}(x,z)+(x-z)^2\Delta^{(34)}(x,z)\big\}w^{v-1}y^{u-1}=G^{(34)}_{u,v}\end{gather*}
for $0\leq u\leq v\leq a$,
\begin{gather*} t^{(34)}(x,z):=rz^{r+s-1}(x-z)+z^{r+s}, \qquad 
\Delta^{(34)}(x,z):=\sum_{k=1}^{r-1}kz^{s+k-1}x^{r-k-1},\end{gather*}
and
\begin{gather*}\sum_r\sum_s uc_{r,u}c_{s,v}\Delta^{(34)}(x,z)y^{u-1}w^{v-1}\\
\qquad{} =\sum_{(i,j)\in J(a,b)}x^iy^j\sum_r\sum_s uc_{r,j+1}c_{s,v}(r-i-1)z^{r+s-i-2}w^{v-1},\end{gather*}
where $(r,s)$ range over $r\geq i+1$ in $D$. Thus, from the third and fourth terms in (\ref{eq401}), we obtain~(\ref{eq22}) and (\ref{eq25}).

On the other hand, we claim
\begin{gather}
 t^{(5)}_k(x,z)+(x-z)^2\Delta_k^{(5)}(x,z) =-\{(p+1)(v-u)-(v-k)s-(k-u)r\}z^{r+s-p}x^p\nonumber\\
 \qquad{} -\{(v-k)s+(k-u)r-p(v-u)\}z^{r+s-p-1}x^{p+1}\label{eq101}
\end{gather}
for $0\leq u\leq v\leq a$, where
\begin{gather*}t_k^{(5)}(x,z):=\{s(v-k)+r(k-u)\}z^{r+s-1}(x-z)+(v-u)z^{r+s},\\
\Delta_k^{(5)}(x,z):=(x-z)^2\sum_{h=1}^p\{(v-k)s+(k-u)r-(v-u)h\}z^{r+s-h-1}x^{h-1},\end{gather*}
where $p$ is a unique integer such that
\begin{gather}\label{eq661}
p(v-u)\leq (v-k)s+(k-u)r <(p+1)(v-u).
\end{gather}
In fact, we see
\begin{gather*}
 t_k^{(5)}(x,z)+(x-z)^2\Delta_k^{(5)}(x,z) =-\{s(v-k)+r(k-u)-l(v-u)\}z^{r+s-l+1}x^{l-1}\\
 \qquad{} + \{s(v-k)+r(k-u)-(l-1)(v-u)\}z^{r+s-l}x^l\\
 \qquad{} +(x-z)^2\sum_{h=l}^p\{(v-k)s+(k-u)r-(v-u)h\}z^{r+s-h-1}x^{h-1}
\end{gather*}
for $l=1,2,\dots,p+1$.

Similarly, we obtain
\begin{gather}
 t^{(5)}_{u+v-k}(x,z)+(x-z)^2\Delta_{u+v-k}^{(5)}(x,z) =-\{(q+1)(v-u)-(v-k)r-(k-u)s\}z^{r+s-q}x^q\nonumber\\
 \qquad{} -\{(v-k)r+(k-u)s-q(v-u)\}z^{r+s-q-1}x^{q+1}\label{eq102}
\end{gather}
for $0\leq u\leq v\leq a$, where
\begin{gather*}t_{u+v-k}^{(5)}(x,z):=\{r(v-k)+s(k-u)\}z^{r+s-1}(x-z)+(v-u)z^{r+s},\\
\Delta_{u+v-k}^{(5)}(x,z):=(x-z)^2\sum_{h=1}^p\{(v-k)r+(k-u)s-(v-u)h\}z^{r+s-h-1}x^{h-1},\end{gather*}
where $q$ is a unique integer such that
\begin{gather}\label{eq662}
q(v-u)\leq (v-k)r+(k-u)s <(q+1)(v-u).
\end{gather}

From (\ref{eq661}) and (\ref{eq662}), we have two possibilities: $p+q=r+s$ and $p+q+1=r+s$. For the former case, (\ref{eq101}) and (\ref{eq102}) are $-(v-u)z^qx^p$ and $-(v-u)z^px^q$, respectively; and for the latter case, (\ref{eq101}) and (\ref{eq102}) are
\begin{gather*}-\{(v-k)s+(k-u)r-q(v-u)\}z^{q+1}x^p+\{(v-k)s+(k-u)r-p(v-u)\}z^{q}x^{p+1}\end{gather*}
and{\samepage
\begin{gather*}-\{(v-k)r+(k-u)s-p(v-u)\}z^{p+1}x^q+\{(v-k)r+(k-u)s-q(v-u)\}z^{p}x^{q+1},\end{gather*}
respectively.}

Since
\begin{gather*}p=\left\lfloor \frac{(v-k)s+(k-u)r}{v-u}\right\rfloor,\qquad
q=\left\lfloor \frac{(v-k)r+(k-u)s}{v-u}\right\rfloor,\\
(v-k)s+(k-u)r-p(v-u)=\underline{\rm mod}\{(v-k)s+(k-u)r,v-u\},\\
(v-k)r+(k-u)s-q(v-u)=\overline{\rm mod}\{(v-k)r+(k-u)s,v-u\},\end{gather*}
we have (\ref{eq181}), which means we obtain (\ref{eq91}) as well
\begin{gather*}
\sum_{k=u+1}^{v-1}\sum_r\sum_s c_{r,u}c_{s,v}\Delta_k^{(5)}(x,z)y^{u+v-k-1}w^{k-1} \\
\qquad {} =\sum_{(i,j)\in J(a,b)}x^iy^j\sum_r\sum_s uc_{r,j+1}c_{s,v} \{(i+1)(v-u)-(j+1-u)s\\
\qquad \hphantom{=\sum_{(i,j)\in J(a,b)}x^iy^j\sum_r\sum_s}{}-(v-j-1)r\}z^{r+s-i-2}w^{v-1},
\end{gather*}
where we have chosen for each $(i,j)\in J(a,b)$, $i=h-1$ and $j=u+v-k-1$, so that $1\leq i+1\leq p$ and $u+1\leq u+v-j-1 \leq v-1$. Hence, $(u,v)$ and $(s,r)$ need to satisfy
\begin{gather*}u\leq j,\qquad j+2\leq v,\qquad (i+1)(v-u)\geq (v-k)s+(k-u)r,
\end{gather*}
respectively. This completes the proof.
\end{proof}

\begin{kei}
\begin{gather*}\lim_{(z,w)\rightarrow (x,y)}R((z,w),(x,y))(x-z)^2=1.\end{gather*}
\end{kei}

\begin{proof}
Let $ f_u:=\sum\limits_{0\leq ar+bu\leq ab, r\geq 0}c_{r,u}x^{r}$. Then, $G^{(12)}_{u,v}$, $G^{(34)}_{u,v}$, and $G^{(5)}_{u,v}$ converge to $uf_uf_v$, $uf_uf_v$, and $-(v-u)(v-u-1)y^{u+v-2}$, respectively. Thus, $G((x,y),(z,w))$ converges to
\begin{gather*}
\sum_{u=0}^auf_u^2y^{2u-2}+\sum_{u=0}^a\sum_{u<v} 2uf_uf_vy^{u+v-2}-\sum_{u=0}^a\sum_{u<v}(v-u)(v-u-1)f_uf_vy^{u+v-2}\\
\qquad{} =\sum_{u=0}^auf_u^2y^{2u-2}+\sum_{u=0}^a\sum_{u<v}\big\{(u+v)-(u-v)^2\big\}f_uf_vy^{u+v-2}\\
\qquad{} =\frac{1}{2}\sum_{u=0}^a\sum_{v=0}^a\big\{u+v-u^2-v^2+2uv\big\}f_uf_vy^{u+v-2}\\
\qquad{} =\sum_{u=0}^a\sum_{v=0}^auvf_uf_vy^{u+v-2}=\sum_{u=1}^auf_uy^{u-1}\sum_{v=1}^avf_vy^{v-1}=\left\{\sum_{u=1}^auf_uy^{u-1}\right\}^2=F_y^2,
\end{gather*}
where
\begin{gather*}\sum_{u=0}^a\sum_{v=0}^auf_uf_vy^{u+v-2}=\sum_{u=1}^ay^{u-2}\sum_{v=0}^af_vy^v=0\end{gather*}
and
\begin{gather*}\sum_{u=0}^a\sum_{v=0}^avf_uf_vy^{u+v-2}=\sum_{u=0}^a\sum_{v=0}^au^2f_uf_vy^{u+v-2}=\sum_{u=0}^a\sum_{v=0}^av^2f_uf_vy^{u+v-2}=0\end{gather*}
have been applied. On the other hand, the denominator of $R((z,w),(x,y))(x-z)^2$ converges to~$F_y^2$ as well. This completes the proof.
\end{proof}

\begin{rei}[generalized hyper-elliptic curves]\rm
For the curve $F(x,y)=y^2+yf_1+f_0=0$,
\begin{gather*}
G((x,y),(z,w))=2yw+f_1w+g_1y+f_1g_1\\
\hphantom{G((x,y),(z,w))=}{}
-\sum_{r\colon \, {\rm odd}}c_{r,0}\big(z^{\frac{r+1}{2}}x^{\frac{r-1}{2}}+z^{\frac{r-1}{2}}x^{\frac{r+1}{2}}\big) -\sum_{r\colon \, {\rm even}}2c_{r,0}z^{\frac{r}{2}}x^{\frac{r}{2}},
\end{gather*}
when
\begin{gather*}
r_{i,0}(z,w)=\sum_{r\geq i+2}\sum_{s\geq 0}c_{r,1}c_{s,1}(r-i-1)z^{r+s-i-2}\\
\hphantom{r_{i,0}(z,w)=}{}
+\sum_{r\geq i+2}c_{r,1}(r-i-1)z^{r-i-2}+\sum_{r=i+2}^{2i+2}c_{r,0}(2i+2-r)
\end{gather*}
and converges to
\begin{gather*}2y^2+2f_1y+f_1^2+2\big(f_1y+y^2\big)=(2y+f_1)^2=F_y^2\end{gather*}
as $(z,w)\rightarrow (x,y)$.
\end{rei}
\begin{rei}[cyclic curves, super-elliptic curves \cite{gal}]
For the curve $F(x,y)=y^a+f_0=0$,
\begin{gather*}G((x,y),(z,w))=
-\sum_r c_{r,0}\sum_{k=1}^{a-1}\big\{\overline{\rm mod}((a-k)r,a)z^{\lceil\frac{(a-k)r}{a}\rceil}x^{\lfloor\frac{kr}{a}\rfloor}\\
\hphantom{G((x,y),(z,w))=-\sum_r c_{r,0}\sum_{k=1}^{a-1}}{}
+\underline{\rm mod}(kr,a)z^{\lfloor\frac{(a-k)r}{a}\rfloor}x^{\lceil\frac{kr}{a}\rceil} \}w^{k-1}y^{a-k-1},
\end{gather*}
when
\begin{gather*}r_{i,j}(z,w)= -\sum_{r=i+2}^b c_{r,0} (ar-a-r-ai-rj) z^{r-2-i}w^{a-2-j}
\end{gather*}
and converges to
\begin{gather*}-2f_0y^{a-2}=F_y^2\end{gather*} as $(z,w)\rightarrow (x,y)$.
\end{rei}
\begin{rei}[trigonal curves \cite{DD,GG}]
For the curve $F(x,y)=y^3-q(x)y-p(x)$ with $p(x)=-\sum\limits_{r=0}^{g+1}c_{r,0}x^r$ and $q(x)=-\sum\limits_{r=0}^{\lfloor\frac{2g+2}{3}\rfloor}c_{r,1}x^r$,
\begin{gather*}
 G((x,y),(z,w)) =\sum_r c_{r,1}x^r\sum_sc_{s,1}z^s+\sum_rc_{r,1}x^rw^2+\sum_sc_{s,1}z^sy^2+3y^2w^2\\
 \hphantom{G((x,y),(z,w)) =}{} -\sum_rc_{r,0}\big\{\overline{\rm mod}(2r,3)z^{\lceil 2r/3\rceil}x^{\lfloor r/3\rfloor}+
\underline{\rm mod}(r,3)z^{\lfloor 2r/3\rfloor}x^{\lceil r/3\rceil}\big\}y\\
\hphantom{G((x,y),(z,w)) =}{}
-\sum_rc_{r,0}\big\{\overline{\rm mod}(r,3)z^{\lceil r/3\rceil}x^{\lfloor 2r/3\rfloor}+ \underline{\rm mod}(2r,3)z^{\lfloor r/3\rfloor}x^{\lceil 2r/3\rceil}
\big\}w\\
\hphantom{G((x,y),(z,w)) =}{}
 -\sum_rc_{r,1}\big\{\overline{\rm mod}(r,2)z^{\lceil r/2\rceil}x^{\lfloor r/2\rfloor}+
\underline{\rm mod}(r,2)z^{\lfloor r/2\rfloor}x^{\lceil r/2\rceil}\big\}wy\\
\hphantom{G((x,y),(z,w))}{}
 = \sum_r c_{r,1}x^r\sum_sc_{s,1}z^s+\sum_rc_{r,1}x^rw^2+\sum_sc_{s,1}z^sy^2+3y^2w^2\\
\hphantom{G((x,y),(z,w)) =}{}
-\sum_{r=3m}c_{r,0}\big\{3z^{2r/3}x^{r/3}y+3z^{2r/3}x^{r/3}w\big\}\\
\hphantom{G((x,y),(z,w)) =}{}
-\sum_{r=3m+1}c_{r,0}\big\{2z^{(2r+1)/3}x^{(r-1)/3}y+z^{(2r-2)/3}x^{(r+2)/3}y\\
\hphantom{G((x,y),(z,w)) =-\sum_{r=3m+1}}{}
+z^{(r+2)/3}x^{(2r-2)/3}w+2z^{(r-1)/3}x^{(2r+1)/3}w\big\}\\
\hphantom{G((x,y),(z,w)) =}{}
 -\sum_{r=3m+2}c_{r,0}\big\{z^{(2r+2)/3}x^{(r-2)/3}y+2z^{(2r-1)/3}x^{(r+1)/3}y\\
\hphantom{G((x,y),(z,w)) =-\sum_{r=3m+2}}{}
+2z^{(r+1)/3}x^{(2r-1)/3}w+z^{(r-2)/3}x^{(2r+2)/3}w\big\}\\
\hphantom{G((x,y),(z,w)) =}{}
 -2\sum_{r=2l}c_{r,1}z^{r/2}x^{r/2}wy\\
\hphantom{G((x,y),(z,w)) =}{}
 -\sum_{r=2l+1}c_{r,1}\big\{z^{(r+1)/2}x^{(r-1)/2}+z^{(r-1)/2}x^{(r+1)/2}\big\}wy,
\end{gather*}
when
\begin{gather*}
r_{i,0}(z,w) = \sum_{r\geq 2i+2}\sum_s c_{r,1}c_{s,1}(r-i-1)z^{r+s-i-2} +\sum_{r\geq 2i+2}c_{r,1}(r-i-1)z^{r-i-2}w^2\\
\hphantom{r_{i,0}(z,w) =}{} +\sum_{r=i+2}^{\lfloor 3(i+1)/2\rfloor}c_{r,0}(3i+3-2r)z^{r-i-2}w
\end{gather*}
and
\begin{gather*}r_{i,1}(z,w)=\sum_{r=i+2}^{3i+3}c_{r,0}(3i+3-r)z^{r-i-2}+\sum_{r=i+2}^{2i+2}c_{r,1}(2i+2-r)z^{r-i-2}w.\end{gather*}
In this sense, Theorem~\ref{theorem1} contains the same formula \cite[p.~203]{GG} as a special case that was described as
\begin{gather*}
r_{i,0}(z,w) = -w^2\partial_zD_z^{i+1}q(z)+wz^{\lfloor i/2 \rfloor}D_z^{\lfloor 3i/2+2\rfloor}\{2z\partial_z p(z)-3(i+1)p(z) \}\\
\hphantom{r_{i,0}(z,w) =}{} +
z^iD_z^{2i+2}\left\{\frac{1}{2}z\partial q(z)^2-(i+1)q(z)^2\right\}
\end{gather*}
and
\begin{gather*}r_{i,1}(z,w)=wz^{i+1}D_z^{2i+3}\{z\partial_zq(z)-2(i+1)q(z)\}+z^{2i+2}D_z^{3i+4}\{z\partial_zp(z)-3(i+1)p(z)\},\end{gather*}
where $D_sp(z)$ denotes $\sum\limits_{k=s}^np_kz^{k-s}$ for polynomial $p(z)=\sum\limits_{k=0}^np_kz^k$.
\end{rei}

All of these applications of the general formula coincide with the known results \cite{DD,GG}.

\section{Open problem}\label{section4}

Let $F$ be a function f\/ield over $\mathbb C$, and $\cal O$ any place of degree one. We def\/ine the set
\begin{gather*}{L}:=\{f\in F\,|\,\operatorname{ord}_Q(f)\geq 0, \,Q\not={\cal O}\}\cup \{0\}\end{gather*}
of functions with poles only at $\cal O$ and the monoid $M:=\{-\operatorname{ord}_{\cal O}(f)|f\in L\}$. We choose generators $a_1,\dots,a_m$ of~$M$ and functions $x_1,\dots,x_m \in {L}$ such that $-\operatorname{ord}_{\cal O}(x_i)=a_i$, $i=1,\dots,m$, where $a_1,\dots,a_m$ are assumed to be mutually prime. For the homomorphism from the polynomial ring ${\mathbb C}[X_1,\dots,X_m]$ to the ring ${\mathbb C}[x_1,\dots,x_m]$ generated by $x_1,\dots,x_m$, we can regard a set of generators of the kernel as the def\/ining equations of the curve~$C$ (Miura~\cite{miura}). The kernel contains $m$ variables and at least $m-1$ equations. Let $d_i$ be the GCM of $a_1,\dots,a_i$. The sequence $a_1,\dots,a_m$ is referred to as telescopic if
\begin{gather*}\frac{a_i}{d_i}\in \left\langle\frac{a_1}{d_{i-1}},\dots,\frac{a_{i-1}}{d_{i-1}}\right\rangle \end{gather*}
for $i=2,\dots,m$. Suzuki~\cite{suzuki} proved that the number of equations is $m-1$ if and only if the sequence can be telescopic by permutating the order of the sequence.

This paper considered only the case $m=2$. Our future work contains solving the problem for telescopic and general Miura curves.

\appendix
\section{Proof of Proposition~\ref{proposition1}}
For the f\/irst part, without loss of generality, we assume $a<b$. For each $0\leq i\leq b-1$, let $f(i)$ be the number of integers $j$'s such that
\begin{gather*}j\leq a-1-\frac{a}{b}(i+1).\end{gather*}
Since
\begin{gather*}f(0)+f(b-2)=\left\lfloor a-1-\frac{a}{b}\right\rfloor +1+0=a-1=f(i)+f(b-i-2)\end{gather*}
for $i=0,\dots,b-2$, and $f(b-1)=0$, where $\lfloor x\rfloor$ is the largest integer no more than $x$, we have
\begin{gather*}\sum_{i=0}^{b-1}f(i)=\frac{1}{2}\sum_{i=0}^{b-2}\{f(i)+f(b-2-i)\}+f(b-1)=\frac{(a-1)(b-1)}{2},\end{gather*}
which means $ \#J(a,b)=\frac{(a-1)(b-1)}{2}$.

For the latter part, since
\begin{gather*}\operatorname{ord}_{\cal O}(dr_{i,j})=\operatorname{ord}_{\cal O}\left(\frac{x^iy^j}{\sum\limits_{k=1}^aky^{k-1}f_k}dx\right)=-(ai+bj)+ab-a-b\geq 0\end{gather*}
for $(i,j)\in J(a,b)$, each $dr_{i,j}$ is holomorphic. On the other hand, since $\operatorname{ord}_{\cal O}(dr_{i,j})\not=\operatorname{ord}_{\cal O}(dr_{i',j'})$ for $(i,j)\not=(i',j')$, the generators $\{dr_{i,j}\}$ consisting of $g$ dif\/ferentials are linearly independent and cover the vector space.

\section{Proof of Proposition~\ref{proposition2}}

From
\begin{gather*}
-w^i\frac{\partial h_j}{\partial w}+jw^{j-1}h_i -\frac{\partial h_i}{\partial y}w^j-jw^{j-1}y^{i-1}\\
\quad{}=-\sum_{k=1}^{j-1}kw^{i+k-1}y^{j-1-k}+j\sum_{k=0}^{i-1}w^{j+k-1}y^{i-1-k}-\sum_{k=1}^{i-1}(i-k)w^{j+k-1}y^{i-k-1}-jw^{j-1}y^{i-1}\\
\quad{} =-\sum_{k=i+1}^{i+j-1}(k-i)w^{k-1}y^{i+j-k-1}+\sum_{k=j+1}^{i+j-1}(k-i)w^{k-1}y^{i+j-k-1}\\
\quad{}=\begin{cases}
\sum\limits_{k=j+1}^{i}(k-i)w^{k-1}y^{i+j-k-1}, &i\geq j+1,\\
0, &i=j,\\
-\sum\limits_{k=i+1}^{j}(k-i)w^{k-1}y^{i+j-k-1},& i\leq j-1,
\end{cases}
\end{gather*}
and
\begin{gather*}
jw^{j-1}h_i-jw^ih_{j-1}-jw^{j-1}y^{i-1}
=
\begin{cases}
j\sum\limits_{k=j+1}^{i}w^{k-1}y^{i+j-1-k}, &i\geq j+1,\\
0, &i=j, \\
-j\sum\limits_{k=i+1}^{j}w^{k-1}y^{i+j-1-k}, &i\leq j-1,
\end{cases}
\end{gather*}
we have
\begin{gather*}
 I(i,j)- \big\{jw^{j-1}y^{i-1}g_i'g_j(x-z)+jw^{j-1}y^{i-1}g_ig_j\big\}\\
\qquad{} = \begin{cases}
\sum\limits_{k=j+1}^{i-1}g_j\{(k-i)g_i'(x-z)+jg_i\}w^{k-1}y^{i+j-k-1}, &i\geq j+1,\\
0, &i=j,\\
-\sum\limits_{k=i+1}^{j-1}g_j\{(k-i)g_i'(x-z)+jg_i\}w^{k-1}y^{i+j-k-1}, &i\leq j-1.
\end{cases}
\end{gather*}
Since $m\leq n$, we have $I(m,n)$ with $i=n$ and $j=m$ for $i\geq j+1$ and $I(m,n)$ with $i=m$ and $j=n$ for $i\leq j-1$, so that~(\ref{eq401}) 
is obtained.

\subsection*{Acknowledgements}
The author would like to thank the anonymous referees. The discussion with them was very helpful for publishing this paper.

\pdfbookmark[1]{References}{ref}
\LastPageEnding


\begin{thebibliography}{99}
\footnotesize\itemsep=0pt

\bibitem{ayano}
Ayano T., Sigma functions for telescopic curves, Ph.D.~Thesis, Osaka
 University, 2013.

\bibitem{baker}
Baker H.F., Abelian functions. Abel's theorem and the allied theory of theta
 functions, \textit{Cambridge Mathematical Library}, Cambridge University Press,
 Cambridge, 1995.

\bibitem{CC}
Buchstaber V.M., Enolski V.Z., Leykin D.V., Hyperelliptic {K}leinian functions
 and applications, in Solitons, Geometry, and Topology: on the Crossroad,
 \textit{Amer. Math. Soc. Transl. Ser.~2}, Vol.~179, Amer. Math. Soc.,
 Providence, RI, 1997, 1--33, \href{http://arxiv.org/abs/solv-int/9603005}{solv-int/9603005}.

\bibitem{EE}
Buchstaber V.M., Enolski V.Z., Leykin D.V., Kleinian functions, hyperelliptic
 {J}acobians and applications, \textit{Rev. Math and Math. Phys.} \textbf{10}
 (1997), no.~2, 1--125.

\bibitem{b}
Buchstaber V.M., Enolski V.Z., Leykin D.V., Rational analogues of abelian
 functions, \href{https://doi.org/10.1007/BF02465189}{\textit{Funct. Anal. Appl.}} \textbf{33} (1999), 83--94.

\bibitem{DD}
Buchstaber V.M., Enolskii V.Z., Leykin D.V., Uniformization of {J}acobi
 varieties of trigonal curves and nonlinear dif\/ferential equations,
 \href{https://doi.org/10.1007/BF02482405}{\textit{Funct. Anal. Appl.}} \textbf{34} (2000), 159--171.

\bibitem{GG}
Buchstaber V.M., Enolski V.Z., Leykin D.V., Multi-dimensional sigma-functions,
 \href{http://arxiv.org/abs/1208.0990}{arXiv:1208.0990}.

\bibitem{BB}
Buchstaber V.M., Leykin D.V., Addition laws on {J}acobians of plane algebraic
 curves, \textit{Proc. Steklov Inst. Math.} \textbf{251} (2005), 49--120.

\bibitem{cf}
Cassels J.W.S., Flynn E.V., Prolegomena to a middlebrow arithmetic of curves of
 genus~{$2$}, \href{https://doi.org/10.1017/CBO9780511526084}{\textit{London Mathematical Society Lecture Note Series}}, Vol.~230, Cambridge University Press, Cambridge, 1996.

\bibitem{onishi}
Eilbeck J.C., Enolski V.Z., Matsutani S., \^Onishi Y., Previato E., Abelian
 functions for trigonal curves of genus three, \href{https://doi.org/10.1093/imrn/rnm140}{\textit{Int. Math. Res. Not.}}
 \textbf{2008} (2008), rnm~140, 38~pages, \href{http://arxiv.org/abs/math.AG/0610019}{math.AG/0610019}.

\bibitem{AA}
Eilbeck J.C., Enolskii V.Z., Leykin D.V., On the {K}leinian construction of
 abelian functions of canonical algebraic curves, in S{IDE}~{III}~--
 Aymmetries and Integrability of Dif\/ference Equations ({S}abaudia, 1998),
 \textit{CRM Proc. Lecture Notes}, Vol.~25, Amer. Math. Soc., Providence, RI,
 2000, 121--138.

\bibitem{FF}
England M., Eilbeck J.C., Abelian functions associated with a cyclic tetragonal
 curve of genus six, \href{https://doi.org/10.1088/1751-8113/42/9/095210}{\textit{J.~Phys.~A: Math. Theor.}} \textbf{42} (2009),
 095210, 27~pages, \href{http://arxiv.org/abs/0806.2377}{arXiv:0806.2377}.

\bibitem{gal}
Galbraith S.D., Paulus S.M., Smart N.P., Arithmetic on superelliptic curves,
 \href{https://doi.org/10.1090/S0025-5718-00-01297-7}{\textit{Math. Comp.}} \textbf{71} (2002), 393--405.

\bibitem{klein1}
Klein F., Ueber hyperelliptische {S}igmafunctionen, \href{https://doi.org/10.1007/BF01445285}{\textit{Math. Ann.}}
 \textbf{27} (1886), 431--464.

\bibitem{klein2}
Klein F., Ueber hyperelliptische {S}igmafunctionen, \href{https://doi.org/10.1007/BF01443606}{\textit{Math. Ann.}}
 \textbf{32} (1888), 351--380.

\bibitem{miura}
Miura S., Error-correcting codes based on algebraic geometry, Ph.D.~Thesis,
 University of Tokyo, 1998.

\bibitem{nakayashiki}
Nakayashiki A., On algebraic expressions of sigma functions for {$(n,s)$}
 curves, \href{https://doi.org/10.4310/AJM.2010.v14.n2.a2}{\textit{Asian~J. Math.}} \textbf{14} (2010), 175--211,
 \href{http://arxiv.org/abs/0803.2083}{arXiv:0803.2083}.

\bibitem{suzuki}
Suzuki J., Miura conjecture on af\/f\/ine curves, \textit{Osaka~J. Math.}
 \textbf{44} (2007), 187--196.

\end{thebibliography}
\end{document}